\documentclass[11pt]{article}
\usepackage{amssymb,amsmath,amsfonts,amsthm}

\usepackage{epsfig,color,mathrsfs}

\usepackage{appendix}

\numberwithin{equation}{section}

\newcommand{\R}{{\mathbb R}}

\newcommand{\cred}{\color{red}}
\definecolor{darkgreen}{rgb}{0,0.7,0.1}

\allowdisplaybreaks

\newtheorem{theorem}{Theorem}[section]

\newtheorem{lemma}[theorem]{Lemma}

\newtheorem{proposition}[theorem]{Proposition}
\newtheorem{remark}[theorem]{Remark}

\begin{document}

\title{\vskip-0.3in Singular solutions for coercive quasilinear elliptic inequalities with nonlocal terms}

\author{Roberta Filippucci\footnote{Dipartimento di Matematica e Informatica, Universit\'a degli Studi di Perugia, Via Vanvitelli 1, 06123 Perugia, Italy; {\tt roberta.filippucci@unipg.it}}
      $\quad$   and    $\quad$
{Marius Ghergu\footnote{School of Mathematics and Statistics,
    University College Dublin, Belfield, Dublin 4, Ireland; {\tt
      marius.ghergu@ucd.ie}}\;\,\footnote{Institute of Mathematics Simion Stoilow of the Romanian Academy, 21 Calea Grivitei St., 010702 Bucharest, Romania}}
}


\maketitle

\begin{abstract}
We study the inequality
$$
{\rm div}\big(|x|^{-\alpha}|\nabla u|^{m-2}\nabla u\big)\geq  (I_\beta\ast u^p)u^q \quad\mbox{ in } B_1\setminus\{0\}\subset \R^N,
$$
where $\alpha>0$, $N\geq 1$, $m>1$,  $p, q>m-1$ and $I_\beta$ denotes the Riesz potential of order $\beta\in(0, N)$. We obtain sharp conditions in terms of these parameters for which positive singular solutions exist. We further establish the asymptotic profile of singular solutions to the double inequality
$$
a(I_\beta\ast u^p)u^q\geq {\rm div}\big(|x|^{-\alpha}|\nabla u|^{m-2}\nabla u\big)\geq  b(I_\beta\ast u^p)u^q \quad\mbox{ in } B_1\setminus\{0\}\subset \R^N,
$$
where $a\geq b>0$ are constants.
\end{abstract}


\medskip



\section{Introduction and the main results}\label{sec1}
In this paper we are concerned with the following quasilinear elliptic inequality
\begin{equation}\label{main}
{\rm div}\big(|x|^{-\alpha}|\nabla u|^{m-2}\nabla u\big)\geq  (I_\beta\ast u^p)u^q \quad\mbox{ in } B_1\setminus\{0\}\subset \R^N,
\end{equation}
and with the double inequality
\begin{equation}\label{main1}
a(I_\beta\ast u^p)u^q\geq {\rm div}\big(|x|^{-\alpha}|\nabla u|^{m-2}\nabla u\big)\geq  b(I_\beta\ast u^p)u^q \quad\mbox{ in } B_1\setminus\{0\},
\end{equation}
where  $\alpha>0$, $\beta\in (0, N)$, $m>1$, $N\geq 1$, $p>0$, $q>m-1$
and $a\geq b>0$.

Throughout this paper, $B_R(z)$ denotes the open ball in $\R^N$, $N\geq 1$, with center at $z\in \R^N$ and having radius $R>0$. When $z=0$, we simply use $B_R$ instead of $B_R(0)$.

The quantity $I_\beta\ast u^p$ represents the convolution operation
$$ (I_\beta\ast u^p)(x)=\int_{B_1} I_\beta(x-y)u^p(y) dy,$$ where $I_\beta:\R^N\to \R$ is the {\it Riesz
potential} of order $\beta\in (0,N)$  given by $$
I_\beta(x)=\frac{A_\beta}{|x|^{N-\beta}}\,,\quad \mbox{ with }\;
A_\beta=\frac{\Gamma\big(\frac{N-\beta}{2}\big)}{\Gamma(\frac{\beta}{2}\big) \pi^{N/2}2^\beta}=
C(N, \beta)> 0. $$
By a positive solution of  \eqref{main} we understand a function
$u\in W^{1,m}_{loc}(B_1\setminus\{0\})\cap C(\overline B_1\setminus\{0\})$
which satisfies:

\begin{itemize}
\item $u>0$, \, $u\in L^p(B_1)$, \,\, ${\rm div}(|x|^{-\alpha}|\nabla u|^{m-2}\nabla u),\,\, (I_\beta*u^p)u^q\in L^{1}_{loc}(B_1\setminus\{0\})$;

\item for any $\phi \in C_{c}^{\infty}(\Omega)$, $\phi\geq 0$ we have
\begin{equation*}
\int_{B_1} |x|^{-\alpha} |\nabla u|^{m-2}\,\nabla u  \cdot \nabla \phi +\int_{B_1}(I_\beta *u^p)u^q \phi\leq 0.
\end{equation*}
\end{itemize}
Solutions of \eqref{main} are called singular if
$$
\displaystyle \limsup_{x\to 0}u(x)=\infty.
$$
\noindent{\bf Remark.} Let us point out that the condition $u\in L^p(B_1)$ is needed to ensure $I_\beta  \ast u^p$ is finite almost everywhere. In fact, these two conditions are equivalent since for $x\in B_1\setminus\{0\}$ we have
\begin{equation}\label{uLp}
\infty>(I_\beta *u^p)(x)=C\int_{B_1}\frac{u^{p}(y)}{|x-y|^{N-\beta }}dy
\geq C\int_{B_1}\frac{u^{p}(y)}{2^{N-\beta }}dy,
\end{equation}
so $u\in L^p(B_1)$. Conversely, if $u\in L^p(B_1)$ then, by standard properties of convolution (see, e.g., \cite[Chapter 2]{LL2001}) one has $I_\beta*u^p\in L^1(B_1)$.

The study of quasilinear elliptic inequalities has received constant attention in the last decades, one general example is the inequality
\begin{equation}\label{noncoercive}
 -{\rm div}[\mathcal{A}(x,u, \nabla u)]\geq f(x,u) \quad\mbox{ in }\Omega,
\end{equation}
which has appeared in many research papers under various structural hypotheses on ${\mathcal{A}}$. The work by Mitidieri and Pohozaev \cite{MP2001} contains many results in this direction and provides the reader with a range of methods to investigate the nonexistence of a solution. The equality case in \eqref{noncoercive} naturally leads to a proper differential equation  and has even a longer history. We only mention here the seminal work of Gidas and Spruck \cite{GS1981} for the semilinear case with power type nonlinearity but also some more recent results \cite{FPR2008}, \cite{GKS2019}, \cite{R2019} dealing with other different situations.

A systematic study of the inequality
$$
 L_\mathcal{A} u= -{\rm div}[\mathcal{A}(x,u, \nabla u)]\geq |x|^{\sigma}u^q \quad\mbox{ in }\Omega,
$$
along with the corresponding system $$ \left\{
\begin{aligned}
&L_\mathcal{A} u= -{\rm div}[\mathcal{A}(x, u, \nabla u)]\geq |x|^au^pv^q\\
&L_\mathcal{B} v= -{\rm div}[\mathcal{B}(x, v, \nabla v)]\geq |x|^b u^rv^s
\end{aligned}
\right.\qquad\mbox{ in } \Omega,
$$
is carried out in \cite{BP2001} for various domains $\Omega\subset \R^N$, such as open balls and their
complements, half balls and half spaces (see also \cite{DM2018} for the case of general nonlinearities). More recently, quasilinear elliptic inequalities and systems integrate the gradient term in the nonlinearity: the authors in \cite{F2009} and \cite{FPR2010} discuss coercive quasilinear inequalities in the form
$$
{\rm div}(g(x)|\nabla u|^{p-2}\nabla u)\geq h(x)f(u)\ell(|\nabla u|)\quad\mbox{ in }\R^N,
$$
and respectively
$$
{\rm div}(h(x) g(u) A(|\nabla u|) \nabla u)\geq f(x, u, \nabla u) \quad\mbox{ in }\R^N.
$$
Systems of quasilinear elliptic inequalities of type
$$
\left\{
\begin{aligned}
& -{\rm div}(h_1(x)A(|\nabla u|)\nabla u)\geq f(x,u,v,\nabla u, \nabla v)\\
& -{\rm div}(h_2(x)B(|\nabla v|)\nabla v)\geq g(x,u,v,\nabla u, \nabla v)
\end{aligned}
\right.\qquad\mbox{ in } \R^N
$$
and
$$
\left\{
\begin{aligned}
& -{\rm div}[\mathcal{A}(x, u, \nabla u)]\geq a(x) u^{p_1}v^{q_1}|\nabla u|^{\theta_1}\\
& -{\rm div}[\mathcal{B}(x, v, \nabla v)]\geq b(x) u^{p_2}v^{q_2}|\nabla u|^{\theta_2}
\end{aligned}
\right.\qquad\mbox{ in } \R^N
$$
are studied in \cite{F2011} and \cite{F2013} respectively.

To the best of our knowledge, the first results dealing with quasilinear elliptic inequalities in the presence of nonlocal terms appear in \cite{CMP2008}. The authors in \cite{CMP2008} obtain
local estimates and Liouville type results for
$$
 -{\rm div}[\mathcal{A}(x,u, \nabla u)]\geq K\ast u^q \quad\mbox{ in }\R^N,
 $$
where $K\in L^{1}_{loc}(\R^N)$, $K\geq 0$ and $q>0$.
Extensions to these results were recently obtained in \cite{GKS2020} in the case $K(x)=|x|^{-\beta}$, $\beta\in (0, N)$.
The related equation
$$
-\Delta u+V(x) u= \big(|x|^{-\beta}\ast u^p\big)u^q
$$
is known in the literature under the name of Choquard (or Choquard-Pekar) equation and arises in various fields ranging from quantum physics to one-component plasma and Newtonian relativity. A survey on the mathematical results on the Choquard equation  is presented in \cite{MV2017}.
Solutions to the Choquard equation featuring isolated singularities are studied in \cite{CZ2016} and \cite{CZ2018}. In \cite{GT2016} and \cite{G2020} it is investigated the
behaviour around the origin of singular solutions to
$$
0\leq -\Delta u\leq \big(I_\alpha\ast u^p\big)u^q\quad\mbox{ in }B_1\setminus\{0\}
$$
and
$$
0\leq -\Delta u\leq \big(I_\alpha\ast u^p\big)   \big(I_\beta\ast u^q\big)  \quad\mbox{ in }B_1\setminus\{0\}.
$$
respectively.
Returning to inequality \eqref{main}, we are now ready to state our first main result.

\begin{theorem}\label{thm1} Assume  $m>1$, $N\geq 1$, $q>m-1$, $\alpha>0$ and $\beta\in (0,N)$.
\begin{enumerate}
\item[\rm (i)] If $N\leq m+\alpha$ then \eqref{main} has always singular solutions.

\item[\rm (ii)] If $N>m+\alpha$ and $p>m-1$ then \eqref{main} has singular solutions if and only if
\begin{equation} \label{cond}
\max\{p,q\}< \frac{N(m-1)}{N-m-\alpha}\,, \,\,\,
p+q <  \frac{(N+\beta)(m-1)}{N-m-\alpha} \,\,\, \mbox{ and }\,\,\, N-2m<2\alpha+\beta.
\end{equation}

\end{enumerate}
\end{theorem}
We next proceed to the study of the double inequality \eqref{main1}.
To formulate our main result on \eqref{main1} we introduce the exponent
\begin{equation}\label{sig}
\sigma=\frac{m+\alpha+\beta}{p+q-m+1}>0.
\end{equation}
Let also
\begin{equation}\label{phiphi}
\Phi_{m,\alpha}(x)=
\left\{
\begin{aligned}
&|x|^{-\frac{N-m-\alpha}{m-1}}&&\quad\mbox{ if }N\neq m+\alpha,\\
&\log\frac{5}{|x|} &&\quad\mbox{ if }N= m+\alpha,
\end{aligned}
\right.
\end{equation}
be the fundamental solution of the weighted $m$-Laplace operator for $m>1$. Note that $\Phi_{m,\alpha}$ satisfies the distributional equality
$$
-{\rm div}\big(|x|^{-\alpha}|\nabla \Phi_{m,\alpha}|^{m-2}\nabla \Phi_{m,\alpha}\big)=c\delta_0\quad\mbox{ in }\mathcal{D}'(\R^N),
$$
for some positive constant $c$.

Given two positive functions $f,g$ defined on $\overline{B}_1\setminus\{0\}$, by $f\asymp g$ we understand that the quotient $f/g$ is bounded  on $\overline{B}_1\setminus\{0\}$ between two positive constants.

In case $\sigma p<N$ we have the following result on \eqref{main1}.
{
\begin{theorem}\label{thm2}  Assume $m>1$, $p, q>m-1$, $\alpha>0$, $\beta\in (0,N)$, $N\geq 1$,
and $\sigma p<N$.
\begin{enumerate}
\item[\rm (i)] {\sf (Existence)}
\begin{enumerate}
\item[\rm (i1)] If $ N>m+\alpha$, then, there exists $a\geq b>0$ and a singular solution of \eqref{main1} if and only if \eqref{cond} holds;
\item[\rm (i2)] If $N\leq m+\alpha$ then \eqref{main1} has always singular solutions for some $a\geq b>0$.
\end{enumerate}
\item[\rm (ii)] {\sf (Asymptotic behavior)} Assume $N\ge m+\alpha$ and
\begin{equation}\label{range_q_asy_beh}
\left\{
\begin{aligned}
&m-1<q<\frac{N-(\sigma p-\beta)^+}{N-m-\alpha}(m-1) &&\quad\mbox{ if \,} N>m+\alpha,\\
&m-1<q<\infty &&\quad\mbox{ if \,}N=m+\alpha.
\end{aligned}
\right.
\end{equation}
\begin{enumerate}
\item[\rm (ii1)] If $\sigma p>\beta$ then any singular solution of \eqref{main1} satisfies
\begin{equation}\label{eqe1}
\mbox{ either}\quad u(x)\asymp \Phi_{m,\alpha}(x)\quad\mbox{ or }\quad u(x)\asymp |x|^{-\sigma}.
\end{equation}
\item[\rm (ii2)] If $\sigma p<\beta$ then any singular solution of \eqref{main1} satisfies
\begin{equation}\label{eqe2}
\mbox{ either}\quad u(x)\asymp \Phi_{m,\alpha}(x) \quad\mbox{ or }\quad u(x)\asymp |x|^{-\frac{m+\alpha}{q-m+1}}.
\end{equation}
\end{enumerate}
\end{enumerate}
\end{theorem}
}
Theorem \ref{thm2}(ii)  above states that any singular solution $u$ of \eqref{main1} either behaves like the fundamental solution $\Phi_{m,\alpha}(x)$ in a neighborhood of the origin or has a  stronger singularity precisely given by \eqref{eqe1}{$_2$}-\eqref{eqe2}{$_2$}. In particular, the asymptotic behaviour in Theorem \ref{thm2}(ii) applies to singular solutions of the equation ${\rm div}\big(|x|^{-\alpha}|\nabla u|^{m-2}\nabla u\big)=(I_\beta\ast u^p)u^q$ in $B_1\setminus\{0\}$.

Our asymptotic behaviour \eqref{eqe1}-\eqref{eqe2} is in line with  \cite[Theorem 1.1]{SYW2016} (see also \cite[Theorem 2.1]{FV1986}) where the authors considered the equation
\begin{equation}\label{veron}
{\rm div}\big(|x|^{-\alpha}|\nabla u|^{m-2}\nabla u\big)=|x|^{-\theta}u^q\quad\mbox{ in }B_1\setminus\{0\},
\end{equation}
for $\theta<m+\alpha$,   $m-1<q<(N-\theta)(m-1)/(N-m-\alpha)$ (if $N>m+\alpha$) and $m-1<q<\infty$ (if $N=m+\alpha$). It is obtained in \cite[Theorem 1.1]{SYW2016}  that any singular solution $u$ of \eqref{veron} satisfies the following behaviour at the origin:
\begin{itemize}
\item either $\, |x|^{\frac{m+\alpha-\theta}{q-m+1}}u(x)\to A$  as $x\to 0$, for some $A=A(N,m,q, \alpha, \theta)>0$;

\item or  $\, \frac{u(x)}{\Phi_{m,\alpha}(x)}\to B$ as $x\to 0$, for some $B=B(N,m,q,\alpha, \theta)>0$ and
$$
 -{\rm div}\big(|x|^{-\alpha}|\nabla u|^{m-2}\nabla u\big)+|x|^{-\theta} u^q=B\delta_0\quad\mbox{ in }\mathcal{D}'(B_1),
$$
where $\delta_0$ denotes the Dirac delta mass concentrated at the origin.
\end{itemize}
In the case of \eqref{main1} such exact behaviour seems difficult to capture due to the presence of the nonlocal term $I_\beta\ast u^p$.

Our approach relies on establishing several a priori estimates for the behavior of the singular solutions to \eqref{main}. These combine the Keller-Osserman type estimates (Proposition \ref{estko}), the Harnack inequality (Propositions \ref{strongh} and \ref{weakh}) and various estimates for the convolution term $I_\beta\ast u^p$. We collect all these results in the next section. Sections 3 and 4 contain the proofs of our main results.

Throughout this paper by $c, C, C_1, C_2,...$ we denote positive generic constants whose values may vary on each occasion. Also, all integrals are computed in the Riemann sense even if we omit the $dx$ or $dy$ symbol.

\section{Preliminary Results}

A key tool in our approach is the use of  a priori estimates for solutions $u\in
W^{1,m}_{loc}(\Omega)\cap C(\Omega)$ of the inequality
\begin{equation}\label{ff}
{\rm div}\big(|x|^{-\alpha}|\nabla u|^{m-2}\nabla u\big) \geq f(x)\quad\mbox{ in }\Omega,
\end{equation}
where $\Omega\subset \R^N$ is an open set and $f\in L^1_{loc}(\Omega)$, $f\geq 0$. Solutions $u$ of \eqref{ff} are understood in the weak
sense, that is, $ {\rm div}\big(|x|^{-\alpha}|\nabla u|^{m-2}\nabla u\big)\in L^{1}_{loc}(\Omega)$ and
\begin{equation}\label{var}
\int_{\Omega} |x|^{-\alpha} |\nabla u|^{m-2}\,\nabla u\cdot \nabla \phi + \int_{\Omega}f(x) \phi\leq 0
\quad\mbox{ for any }\phi \in C_{c}^{\infty}(\Omega),\phi\geq 0.
\end{equation}
In \cite[Proposition 2.1]{BGY2013}  the authors obtain a priori estimates for solutions to the general quasilinear inequality
$$
{\rm div}[\mathcal{A}(x,u, \nabla u)]\geq f(x) \quad\mbox{ in }\Omega,
$$
where $\mathcal{A}$ is weakly-$m$-coercive. A careful analysis of the proof of \cite[Proposition 2.1]{BGY2013} reveals that the same arguments cand be employed for
$$
{\rm div}[|x|^{-\alpha} \mathcal{A}(x,u,\nabla u)]\geq f(x) \quad\mbox{ in }\Omega,
$$
which contains as a particular case the inequality \eqref{ff}.
The result below is a reformulation of \cite[Proposition 2.1]{BGY2013}.
\begin{proposition}\label{p1}
Let $\Omega\subset \R^N$ be an open set such that $B_{4R}\setminus
B_{R/2}\subset \Omega$ for some $R>0$. Let $u\in C(\Omega)\cap  W^{1,1}_{loc}(\Omega)$ be a positive solution of
\eqref{ff}.

Take $\phi\in C^\infty_c(\Omega)$ such that $0\leq \phi\leq 1$ and
\begin{itemize}
\item supp$\,\phi\subset B_{4R}\setminus B_{R/2}$;
\item  $\phi=1$ in $B_{2R}\setminus B_{R}$;
\item  $|\nabla \phi|\leq \frac{C}{R}$ in $\Omega$.
\end{itemize}
Then, for any $\ell> m-1$ there exists $\Lambda=\Lambda(m,\ell)$  such that for any $\lambda>\Lambda$ there exists $C>0$ independent of $R$ with
\begin{equation}\label{mainest1}
\int_{\Omega} f(x) \phi^\lambda \leq CR^{N-m-\alpha-\frac{m-1}{\ell}N}
\Big(\int_\Omega u^\ell\phi^\lambda \Big)^{\frac{m-1}{\ell}}.
\end{equation}

\end{proposition}

In the next results we recall the strong and the weak Harnack inequality for the weighted $m$-Laplace operator.

\begin{proposition}\label{strongh}{\rm (Strong Harnack inequality)}

Let $u\in W^{1,m}_{loc}(\Omega)\cap C(\Omega)$, $u\geq 0$ satisfy
$$
{\rm div}\big(|x|^{-\alpha}|\nabla u|^{m-2}\nabla u\big)+a(x)u^{m-1}=0\quad\mbox{ in }\Omega,
$$
where $|a(x)|\leq c |x|^{-m-\alpha}$ for some constant $c>0$. Assume $x\in \Omega$ and $r>0$ are such that $B_{3r}(x)\subset \Omega$. Then, there exists a constant $C>0$ independent of $u$ such that
\begin{equation}\label{harns}
\max_{\overline B_{r}(x)} u\leq C\min_{\overline B_{r}(x)} u.
\end{equation}
\end{proposition}
\begin{proof} Note that $u$ satisfies the equation
$$
{\rm div}\big(|\nabla u|^{m-2}\nabla u\big)-\frac{\alpha}{|x|^2} |\nabla u|^{m-2}\nabla u\cdot x + b(x)u^{m-1}=0\quad\mbox{ in }\Omega,
$$
where $b(x)=a(x)|x|^{\alpha}$ and $|b(x)|\leq c|x|^{-m}$. The above equation
fulfills the structural assumptions in \cite[Theorem 1.1]{T1967}. According to this result, $u$ satisfies \eqref{harns}.
\end{proof}

\begin{proposition}\label{weakh}{\rm (Weak Harnack inequality)}

Let $R>0$ and $a,b,c$ be real numbers such that $a>b>3c>0$. Assume $\Omega\subset \R^N$ is an open set such that
$$
\overline B_{(a+3c)R}\setminus B_{(b-3c)R} \subset \Omega.
$$
Suppose $u\in W^{1,m}_{loc}(\Omega)\cap C(\Omega)$ satisfies $u\geq 0$ and
\begin{equation}\label{eqwew}
{\rm div}\big(|x|^{-\alpha}|\nabla u|^{m-2}\nabla u\big) \geq 0\quad\mbox{ in } \Omega.
\end{equation}
Then, for any $\ell>m-1$, there exists a constant $C>0$ independent of $R$ such that
\begin{equation}\label{whi}
R^{N/\ell} \sup_{B_{aR}\setminus B_{bR}} u \leq C \Big(\int_{B_{(a+2c)R}\setminus B_{(b-2c)R}}u^\ell\Big)^{1/\ell}.
\end{equation}
\end{proposition}
\begin{proof}
Observe first that \eqref{eqwew} is equivalent to
$$
{\rm div}\big(|\nabla u|^{m-2}\nabla u\big)-\frac{\alpha}{|x|^2} |\nabla u|^{m-2}\nabla u\cdot x \geq 0\quad\mbox{ in }\Omega,
$$
which satisfies the structural assumptions in \cite{T1967}.

Let $z_1, z_2,\dots, z_k\in \Omega$ be such that $\{B_{cR}(z_i)\}_{1\leq i\leq k}$ is a finite cover with open balls of the compact set $\overline B_{aR}\setminus B_{bR}$. By the standard Harnack inequality (see Trudinger \cite[Theorem 1.3]{T1967}) we find
$$
R^{N/\ell} \sup_{B_{cR}(z_i)} u \leq C \Big(\int_{B_{2cR}(z_i)} u^\ell\Big)^{1/\ell} \leq C \Big(\int_{B_{(a+2c)R}\setminus B_{(b-2c)R}}u^\ell\Big)^{1/\ell}.
$$
Thus,
$$
R^{N/\ell} \sup_{B_{aR}\setminus B_{bR}} u\leq
R^{N/\ell} \sup_{\cup_{i=1}^k B_{cR}(z_i)} u \leq C\Big(\int_{B_{(a+2c)R}\setminus B_{(b-2c)R}}u^\ell\Big)^{1/\ell}.
$$
\end{proof}

\begin{proposition}\label{estg}
Assume
$u\in W^{1,m}_{loc}(B_1\setminus\{0\})\cap C(B_1\setminus\{0\})$ satisfies $u\geq 0$ and
$$
{\rm div}(|x|^{-\alpha}|\nabla u|^{m-2}\nabla u)\geq 0\quad\mbox{ in } B_1\setminus\{0\}.
$$
Then, either $u$ is bounded near the origin, or there exist $C>0$ and $r_0\in (0,1/2)$ such that
\begin{equation}\label{eqeq}
\sup_{|x|=r} \frac{u(x)}{\Phi_{m,\alpha}(x)} \geq C  \quad\mbox{ for all }r\in (0,r_0),
\end{equation}
where $\Phi_{m,\alpha}$ is defined in \eqref{phiphi}.
\end{proposition}
\begin{proof} Assume that \eqref{eqeq} does not hold. Hence,
$$
\liminf_{r\to 0} \Big(\sup_{|x|=r} \frac{u(x)}{\Phi_{m,\alpha}(x)}\Big)=0.
$$
Then, for any $k\geq 1$ there exists $r_k\in (0,1/2)$, with  $r_k\to0$ as $k\to\infty$, such that
$$
\sup_{|x|=r_k} \frac{u(x)}{\Phi_{m,\alpha}(x)}\leq \frac{1}{k}\quad\mbox{ for all }k\geq 1.
$$
A comparison principle in the annular region $B_{1/2}\setminus B_{r_k}$ shows that for all $k\geq 1$ we have
$$
u(x)\leq \frac{1}{k}\Phi_{m,\alpha}(x)+\max_{|x|=1/2}u(x)\quad\mbox{ in }B_{1/2}\setminus B_{r_k},
$$
Letting $k\to \infty$ in the above estimate we deduce that $u$ is bounded in the ball $B_{1/2}$.
\end{proof}
The result below provides a first important estimate for solutions to \eqref{main}.

\begin{proposition}\label{estko}{\rm (Keller-Osserman type estimates)}

Assume $p+q>2(m-1)$ and let $u\in W^{1,m}_{loc}(B_1\setminus\{0\})\cap C(B_1\setminus\{0\})$ be a positive solution of \eqref{main}. Then,  there exist $C>0$ such that
\begin{equation}\label{ttau}
u(x) \leq C |x|^{-\sigma} \quad\mbox{ in } B_1\setminus\{0\},
\end{equation}
where $\sigma>0$ is given by \eqref{sig}.
\end{proposition}
\begin{proof}
We use Proposition \ref{p1} with $\Omega=B_1$, $R\in (0,1/4)$, $f(x)=(I_\beta\ast u^p)u^q$ and $\ell=(p+q)/2>m-1$. From \eqref{mainest1}  we find
{  \begin{equation}\label{fe1}
C R^{N-m-\alpha-\frac{m-1}{\ell}N}
\Big(\int_{B_1} u^\ell\phi^\lambda \Big)^{\frac{m-1}{\ell}}\geq
\int_{B_1} (I_\beta\ast u^p)u^q \phi^\lambda,
\end{equation}}
where $\phi\in C_c^\infty(B_1\setminus\{0\})$ and $\lambda>m$ are chosen as in Proposition \ref{p1}.
If  $x,y\in B_{{2R}}\subset \text{supp }\phi$, then $|x-y|\leq |x|+|y|\leq 4R$ so
$$
\begin{aligned}
(I_\beta*u^p)(x)&\geq C\int_{B_{4R}}\frac{u^{p}(y)}{|x-y|^{N-\beta}}dy\\
&\geq C\int_{B_{2R}}\frac{u^{p}(y)}{(4R)^{N-\beta}}dy\\
&\geq CR^{\beta-N}\int_{B_1}u^{p}(y)\phi^\lambda(y) dy,
\end{aligned}
$$
since $0\le\phi\le1$.
Using this fact in \eqref{fe1} together with H\"older's inequality, for $\ell=(p+q)/2$ we find
$$
CR^{\tau} \Big(\int_{B_1} u^\ell\phi^\lambda \Big)^{\frac{m-1}{\ell}}\geq \Big(\int_{B_1} u^p\phi^\lambda \Big)
\Big(\int_{B_1} u^q\phi^\lambda \Big)\geq \Big(\int_{B_1} u^\ell\phi^\lambda \Big)^{2},
$$
where
\begin{equation}\label{ta}
\tau=2N-m-\alpha-\beta-\frac{m-1}{\ell}N.
\end{equation}
Now, using the fact that $\phi=1$ in $B_{2R}\setminus B_R$ and the weak Harnack inequality \eqref{whi} with $a=7/4$, $b=5/4$ and $c=1/8$ we deduce
$$
\begin{aligned}
CR^\tau&\geq  \Big(\int_{B_1} u^\ell\phi^\lambda \Big)^{2-\frac{m-1}{\ell}}\geq \Big(\int_{B_{2R}\setminus B_R} u^\ell  \Big)^{2-\frac{m-1}{\ell}}\\
&\geq\Big( R^N\sup_{\frac{5R}{4}<|x|<\frac{7R}{4}} u^\ell\Big)^{2-\frac{m-1}{\ell}}\\
&\geq R^{2N-\frac{m-1}{\ell}N} \Big( \sup_{\frac{5R}{4}<|x|<\frac{7R}{4}} u^{p+q-m+1} \Big).
\end{aligned}
$$
From here and \eqref{ta} we derive \eqref{ttau}.
\end{proof}
Similar to Proposition \ref{estko} we have:
\begin{proposition}\label{lem1}
Let $\theta\geq 0$ and $q>\max\{m-1,\theta\}$.

If $u\in W^{1,m}_{loc}(B_1\setminus\{0\})\cap C(\overline B_1\setminus\{0\})$ is positive and satisfies
\begin{equation}\label{eqw1}
{\rm div}(|x|^{-\alpha}|\nabla u|^{m-2}\nabla u) \geq |x|^{-\theta}u^q\quad\mbox{ in }B_1\setminus\{0\},
\end{equation}
then
\begin{equation}\label{eqw2}
w(x)\leq C|x|^{-\frac{m+\alpha-\theta}{q-m+1}}\quad\mbox{ for all }x\in B_1\setminus\{0\},
\end{equation}
for some constant $C>0$.
\end{proposition}
\begin{proof}  According to \eqref{mainest1} with $\ell=q>m-1$ we have
$$
C R^{N-m-\alpha-\frac{m-1}{q}N}
\Big(\int_{B_1} u^q\phi^\lambda \Big)^{\frac{m-1}{q}}\geq
\int_{B_1} |x|^{-\theta}u^q \phi^\lambda.
$$
Since supp$\,\phi\subset B_{4R}\setminus B_{R/2}$, from the above estimate and the weak Harnack inequality \eqref{whi}
with $a=7/4$, $b=5/4$ and $c=1/8$ it follows that
$$
\begin{aligned}
C R^{N-m-\alpha-\frac{m-1}{q}N} &
\geq R^{-\theta}\Big(\int_{B_{2R}\setminus B_R} u^q \Big)^{1-\frac{m-1}{q}}\\
&\geq R^{-\theta}\Big( R^N\sup_{\frac{5R}{4}<|x|<\frac{7R}{4}} u^q\Big)^{1-\frac{m-1}{\ell}}\\
&\geq R^{N-\theta-\frac{m-1}{q}N} \Big( \sup_{\frac{5R}{4}<|x|<\frac{7R}{4}} u^{q-m+1} \Big).
\end{aligned}
$$
From here, we easily deduce \eqref{eqw2}.
\end{proof}

{
\begin{lemma}\label{lem2} {\rm (See \cite[Theorem 1.1]{SYW2016})}
Let $m>1$, $N\geq m+\alpha>\theta$ and
\begin{equation}\label{range_q_theta}
\left\{
\begin{aligned}
&m-1<q<\frac{(N-\theta)(m-1)}{N-m-\alpha} &&\quad\mbox{ if \,}N>m+\alpha,\\
&m-1<q<\infty &&\quad\mbox{ if \,}N=m+\alpha.
\end{aligned}
\right.
\end{equation}
Let $u\in W^{1,m}(B_1\setminus\{0\})\cap C(B_1\setminus\{0\})$, $u\geq 0$, be a singular solution of
\begin{equation}\label{eqte}
{\rm div}(|x|^{-\alpha}|\nabla w|^{m-2}\nabla w)=|x|^{-\theta}w^q\quad\mbox{ in }B_1\setminus\{0\}.
\end{equation}
Then,
\begin{center}
either $\; w\asymp \Phi_{m,\alpha}(x)\;\;$ or $\;\; w\asymp |x|^{-\frac{m+\alpha-\theta}{q-m+1}}$.
\end{center}
\end{lemma}
}

\begin{proposition}\label{vv} Assume $N>m+\alpha$ and $q\geq \frac{N(m-1)}{N-m-\alpha}$. Then, any solution of \eqref{main} is bounded around the origin.
\end{proposition}

\begin{proof} We use some tools from \cite[Proposition 1.2]{VV1980}. Let
$$
\nu=\frac{N(m-1)}{N-m-\alpha}
$$
and let $u$ be a positive solution of \eqref{main}. We note that since $u\in L^p(B_1)$, $u$ satisfies
\begin{equation}\label{condx0}
{\rm div}(|x|^{-\alpha}|\nabla u|^{m-2}\nabla u)\geq  c\, u^{q}\quad\mbox{ in }B_1\setminus\{0\},
\end{equation}
where $c=2^{\alpha-N}\int_{B_1}u^p>0$, by \eqref{uLp}. Using Proposition \ref{lem1} (with $\theta=0$ and being $q\geq \nu>m-1$) we deduce
$$
u(x)\leq C|x|^{-\frac{m+\alpha}{q-m+1}}\quad\mbox{ in }B_1\setminus\{0\}.
$$
In particular, again by $q\geq \nu$, it follows that
\begin{equation}\label{condx1}
u(x)\leq C|x|^{-\frac{m+\alpha}{\nu-m+1}}\quad\mbox{ in }B_1\setminus\{0\}.
\end{equation}
Also, from \eqref{condx0} we deduce
\begin{equation}\label{condx2}
{\rm div}(|x|^{-\alpha}|\nabla u|^{m-2}\nabla u)\geq cu^{\nu}-C\quad\mbox{ in }B_1\setminus\{0\},
\end{equation}
for some $C>0$.

In order to proceed to the proof of Proposition \ref{vv} we need the following result.

\begin{lemma}\label{lx}
Assume $u$ satisfies \eqref{condx2}. Then, for any $\phi\in C^1_c(B_1\setminus\{0\})$, $\phi\geq 0$ and any number $M\geq (C/c)^\nu$ we have
\begin{equation}\label{condx3}
\Big\||x|^{-\alpha/m}\phi|\nabla (u-M)^+|\Big\|_{L^m(B_1)}\leq m
\Big\||x|^{-\alpha/m} (u-M)^+ |\nabla \phi |\Big\|_{L^m(B_1)}.
\end{equation}
\end{lemma}
\begin{proof}[Proof of Lemma \ref{lx}] Let $\{\eta_k\}\subset C^1(\R)$ be such that $\eta_k\ge0$,
$$
\eta_k'=0\mbox{ on }(-\infty, 0), \quad \eta_k'>0\mbox{ on }(0,\infty),
$$
$$
\eta_k'(t)\to {\rm sign}^+(t), \quad \eta_k(t)\to t^+\quad\mbox{ as }k\to \infty,
$$
where ${\rm sign}^+(t)=1$ if $t>0$ and ${\rm sign}^+(t)=0$ if $t<0$.
Take $\eta_k(u-M)\phi^m$ as a test function in \eqref{condx2}. We find
$$
\int_{B_1} (cu^\nu-C)\eta_k(u-M)\phi^m dx+\int_{B_1} |x|^{-\alpha}|\nabla u|^{m-2}\nabla u\cdot \nabla\big(\eta_k(u-M)\phi^m\big) dx\leq 0.
$$
Since $(cu^\nu-C)\eta_k(u-M)\phi^m\geq 0$, by the choice of $M$, it follows that
$$
\int_{B_1} |x|^{-\alpha}\phi^m\eta_k'(u-M)  |\nabla u|^{m-2}\nabla u\cdot \nabla(u-M)^+  dx + m\int_{B_1} |x|^{-\alpha} \phi^{m-1}\eta_k(u-M) |\nabla u|^{m-2} \nabla u\cdot \nabla \phi dx\leq 0.
$$
Letting $k\to \infty$, by Fatou's lemma we find
$$
\int_{B_1} |x|^{-\alpha} \phi^m  |\nabla (u-M)^+|^{m} dx +m
\int_{B_1} |x|^{-\alpha} \phi^{m-1} (u-M)^+ |\nabla u|^{m-2} \nabla u\cdot \nabla \phi dx\leq 0,
$$
so
\begin{equation}\label{condx4}
\int_{B_1} |x|^{-\alpha} \phi^m  |\nabla (u-M)^+|^{m} dx \leq m
\int_{B_1} |x|^{-\alpha} \phi^{m-1} (u-M)^+ |\nabla u|^{m-1} |\nabla \phi| dx.
\end{equation}
By H\"older's inequality and since  $(u-M)^+ |\nabla u|=(u-M)^+ |\nabla (u-M)^+|$,  we estimate the right hand-side of \eqref{condx4} as
\begin{equation}\label{condx5}
\begin{aligned}
\int_{B_1} |x|^{-\alpha} \phi^{m-1} (u-M)^+ |\nabla u|^{m-1} |\nabla \phi| dx  \leq &
\Big( \int_{B_1} |x|^{-\alpha} \phi^m  |\nabla (u-M)^+|^{m} dx \Big)^{1/m'}\times\\
&\times \Big(\int_{B_1}|x|^{-\alpha}|\nabla \phi|^m|(u-M)^+|^m dx\Big)^{1/m},
\end{aligned}
\end{equation}
where $m'$ is the H\"older conjugate of $m$. Using \eqref{condx5} into \eqref{condx4} we deduce \eqref{condx3}.
\end{proof}
We are now ready to proceed to the proof of Proposition \ref{vv} whose arguments will be divided into two steps.

\noindent{\bf Step 1:} $u\in L^\nu_{loc}(B_1)$. Let $\eta\in C^1(\R)$ be such that $\eta\geq 0$, $\eta$ is bounded, $\eta=0$ on $(-\infty, 0)$ and $\eta'>0$ on $(0,\infty)$. Let also $\{\zeta_k\}\in C^1_c(\R^N)$ be such that
$$
\zeta_k(x)=\left\{
\begin{aligned}
0 \quad&\mbox{ if }|x|<\frac{1}{2k} \mbox{ or }|x|>\frac{2}{3},\\
1 \quad&\mbox{ if }\frac{1}{k}<|x|<\frac{1}{2},
\end{aligned}
\right.
\quad\mbox{ and }\quad|\nabla \zeta_k|\leq Ck.
$$
Define $A_k=B_{1/k}\setminus B_{1/(2k)}$ and
$$
M\geq \max\big\{(C/c)^q, \max_{1/2\leq |x|\leq 2/3}u(x)\big\}.
$$
We next test \eqref{condx2} with $\zeta_k\eta(u-M)$. We find
$$
\int_{B_1} (cu^\nu-C)\zeta_k \eta (u-M) dx+\int_{B_1} |x|^{-\alpha}|\nabla u|^{m-2}\nabla u\cdot \nabla\big(\zeta_k \eta(u-M) \big) dx\leq 0.
$$
Since $\eta'\geq 0$ and $|\nabla u|^{m-2}\nabla u\nabla (u-M)= |\nabla u|^m\ge0$, it follows that
\begin{equation}\label{thetaa}
\int_{B_1} (cu^\nu-C)\zeta_k \eta (u-M) dx\leq \Gamma_k:=
\int_{B_1} |x|^{-\alpha}\eta(u-M)|\nabla u|^{m-1}|\nabla \zeta_k| dx.
\end{equation}
Observe that $\eta(u-M)\nabla \zeta_k=0$ outside of $A_k$, being $M\ge \max_{1/2\leq |x|\leq 2/3}u(x)$.
Using the fact that $\eta$ is bounded together with H\"older's inequality we find
\begin{equation}\label{condx6}
\begin{aligned}
\Gamma_k &\leq \|\eta\|_\infty\int_{A_k}  |x|^{-\alpha} |\nabla (u-M)^+|^{m-1}|\nabla \zeta_k| dx\\
&\leq C
\Big\||x|^{-\alpha/m}|\nabla (u-M)^+|\Big\|_{L^m(A_k)}^{m-1}
\Big\||x|^{-\alpha/m} |\nabla \zeta_k |\Big\|_{L^m(A_k)}.
\end{aligned}
\end{equation}
By the definition of $\zeta_k$ and the fact that $|\nabla \zeta_k|\leq ck$ we have
$$
\Big\| |x|^{-\alpha/m} |\nabla \zeta_k |\Big\|_{L^m(A_k)}\leq C k^{1-\frac{N-\alpha}{m}}.
$$
Using this fact in \eqref{condx6} together with $\zeta_{2k}=1$ in $A_k$ and $\zeta_k\ge0$ in $A_{2k}$, we further estimate
\begin{equation}\label{condx7}
\begin{aligned}
\Gamma_k &\leq C k^{1-\frac{N-\alpha}{m}} \Big\||x|^{-\alpha/m}|\nabla (u-M)^+|\Big\|_{L^m(A_k)}^{m-1}\\
& \leq C k^{1-\frac{N-\alpha}{m}} \Big\||x|^{-\alpha/m}\zeta_{2k} |\nabla (u-M)^+|\Big\|_{L^m(A_{2k}\cup A_k)}^{m-1}\\
& \leq C k^{1-\frac{N-\alpha}{m}} \Big\||x|^{-\alpha/m}(u-M)^+|\nabla \zeta_{2k} | \Big\|_{L^m(A_{2k})}^{m-1},
\end{aligned}
\end{equation}
where in the last inequality we have used \eqref{condx3} with $\phi=\zeta_{2k}$ and the fact that $\nabla \zeta_{2k}=0$ in $A_k$.
From \eqref{condx1} we have
$$
\int_{A_{2k}}|x|^{-\alpha}|(u-M)^+|^m|\nabla \zeta_{2k}|^{m}dx\leq Ck^{\alpha+m}\int_{A_{2k}}
|(u-M)^+|^m\leq Ck^{\alpha+m-N+\frac{m(m+\alpha)}{\nu-m+1}}.
$$
Hence, from \eqref{condx7} we deduce
$$
\Gamma_k\leq Ck^{1-\frac{N-\alpha}{m}+\big(\alpha+m-N+\frac{m(m+\alpha)}{\nu-m+1}  \big)\frac{m-1}{m}}=C.
$$
We now replace $\eta$ in \eqref{thetaa} by a sequence $\{\eta_n\}$ such that $\eta_n(t)\to {\rm sign}^+(t)$ as $n\to \infty$.
Letting $n\to \infty$ and then $k\to \infty$ in \eqref{thetaa}, since  supp $\zeta_k=\overline{B}_{2/3}$ and $\zeta_k\to 1$  in $B_{1/2}$,  we find
$$
\int_{B_1} (cu^\nu-C){\rm sign}^+(u-M) dx\leq C,
$$
so $u\in L^\nu_{loc}(B_1)$.
\medskip

\noindent{\bf Step 2:} $u\in L^\infty_{loc}(B_1)$. We return to the estimate \eqref{condx7} and split our analysis into two cases.
\begin{itemize}
\item Case 2.1: $\nu\geq m$.  By H\"older's inequality we find
$$
\begin{aligned}
\Big\||x|^{-\alpha/m}(u-M)^+|\nabla \zeta_{2k} | \Big\|_{L^m(A_{2k})}&\leq Ck^{\frac{\alpha}{m}+1} \big\|(u-M)^+ \big\|_{L^m(A_{2k})}\\
&\leq Ck^{\frac{\alpha}{m}+1}\big\|(u-M)^+\big\|_{L^\nu(A_{2k})} |A_{2k}|^{\frac{1}{m}-\frac{1}{\nu}}\\
&= C k^{\frac{\alpha}{m}+1-N\big(\frac{1}{m}-\frac{1}{\nu}\big)}o(1) \;\mbox{ as }k\to \infty.
\end{aligned}
$$
Using this estimate in \eqref{condx7} we deduce
$\Gamma_k \le k^{\frac{N(m-1)}{\nu}-N+m+\alpha}o(1)=o(1)$ as $k\to \infty$, thanks to the value of $\nu$.

\item Case 2.2: $\nu< m$.  From \eqref{condx1} we have
$$
\begin{aligned}
\Big\||x|^{-\alpha/m}(u-M)^+|\nabla \zeta_{2k} | \Big\|_{L^m(A_{2k})}&\leq Ck^{\frac{\alpha}{m}+1} \big\|(u-M)^+ \big\|_{L^m(A_{2k})}\\
&\leq Ck^{\frac{\alpha}{m}+1} \sup_{A_{2k}} |(u-M)^+|^{1-\frac{\nu}{m}} \big\|(u-M)^+ \big\|_{L^\nu(A_{2k})}^{\frac{\nu}{m}} \\
&= Ck^{\frac{\alpha}{m}+1} \sup_{A_{2k}} |(u-M)^+|^{1-\frac{\nu}{m}} o(1) \\
&\leq C k^{\frac{\alpha}{m}+1+\frac{m+\alpha}{\nu-m+1}\big(1-\frac{\nu}m\big)}
o(1)\;\mbox{ as }k\to \infty,
\end{aligned}
$$
and from \eqref{condx7} we again derive $\Gamma_k \le C k^{[N(m-1)-\nu(N-m-\alpha)]/m}o(1)=o(1)$ as $k\to \infty$,
 thanks to the value of $\nu$.
\end{itemize}

\noindent We now return to \eqref{thetaa} and let $k\to \infty$ to deduce
$$
\int_{B_{1/2}} (cu^\nu-C) \eta (u-M) dx=0.
$$
Since $\eta\geq 0$, it follows that $u\leq M$ in $B_{1/2}$, so $u\in L^\infty_{loc}(B_1)$ which completes our proof.
\end{proof}
\begin{lemma}\label{gt2015}
Let $a,b\in (0,N)$ and $\theta\geq 0$. Then, there exists $C>c>0$ such that:
\begin{enumerate}
\item[\rm (i)] If $a+b>N$ one has
\begin{equation}\label{leqe1}
\frac{c\Big(\displaystyle \log\frac{5}{|x|}\Big)^{-\theta}}{|x|^{a+b-N}}\leq
\displaystyle \int_{|y|<1} \frac{\Big(\displaystyle \log\frac{5}{|y|}\Big)^{-\theta} dy}{|x-y|^a|y|^b} \leq \frac{C\Big(\displaystyle \log\frac{5}{|x|}\Big)^{-\theta}}{|x|^{a+b-N}}\quad\mbox{ for all }x\in B_1\setminus\{0\}.
\end{equation}

\item[\rm (ii)] If $a+b=N$, $\theta\neq 1$, one has
\begin{equation}\label{leqe2}
c\Big(\displaystyle \log\frac{5}{|x|}\Big)^{(1-\theta)^+} \leq
\displaystyle \int_{|y|<1} \frac{\Big(\displaystyle \log\frac{5}{|y|}\Big)^{-\theta} dy}{|x-y|^a|y|^b} \leq C \Big(\displaystyle \log\frac{5}{|x|}\Big)^{(1-\theta)^+}
\quad\mbox{ for all }x\in B_1\setminus\{0\}.
\end{equation}

\item[\rm (iii)] If $a+b<N$ one has
\begin{equation}\label{leqe3}
c\leq
\displaystyle \int_{|y|<1} \frac{\Big(\displaystyle \log\frac{5}{|y|}\Big)^{-\theta} dy}{|x-y|^a|y|^b} \leq C
\quad\mbox{ for all }x\in B_1\setminus\{0\}.
\end{equation}

\end{enumerate}

\end{lemma}

The proof of the above lemma will be given in the Appendix.

\begin{remark}\label{remk}
A direct and useful calculation shows that if
$$
u(x)=\kappa |x|^{-\gamma}\Big(\log\frac{5}{|x|}\Big)^{-\tau}\,,\gamma>0,
$$
then
$$
\begin{aligned}
{\rm div}\big(|x|^{-\alpha}|\nabla u|^{m-2}\nabla u\big)=&\kappa^{m-1}|x|^{-\gamma(m-1)-m-\alpha}\Big(\log\frac{5}{|x|}\Big)^{-\tau(m-1)-m}\times\\
&\times \Big| -\gamma \log\frac{5}{|x|}+\tau \Big|^{m-2}\Big[A\Big(\log\frac{5}{|x|}\Big)^{2}+B \log\frac{5}{|x|}+C  \Big],
\end{aligned}
$$
where
\begin{equation}\label{abc}
\begin{aligned}
A=&\gamma\big[\gamma(m-1)-(N-m-\alpha)\big],\\
B=&\tau\big[-2\gamma (m-1)+(N-m-\alpha)\big],\\
C=&(m-1)\tau(\tau+1).
\end{aligned}
\end{equation}
\end{remark}

\section{Proof of Theorem \ref{thm1}}

\noindent (i) {  Let
\begin{equation}\label{gamm0}
0<\gamma<\min\left\{\frac{\beta}{p}, \frac{m+\alpha}{q-m+1}  \right\}.
\end{equation}
We show that $u(x)=\kappa |x|^{-\gamma}$ is a singular radially symmetric solution of \eqref{main} for suitable $\kappa\in (0,1)$. Since from \eqref{gamm0} we have $p\gamma<\beta<N$ it follows that $u\in L^p(B_1)$.
By Remark \ref{remk} (in which we take $\tau=0$) one has
\begin{equation}\label{gamm2}
{\rm div}\big(|x|^{-\alpha}|\nabla u|^{m-2}\nabla u\big)=\kappa^{m-1}\gamma^{m-2} A |x|^{-(m-1)\gamma-m-\alpha} \quad\mbox{ in }B_1\setminus\{0\},
\end{equation}
where $A $ is defined in \eqref{abc}$_1$. From $N\le m+\alpha$ and $\gamma>0$ we have $A>0$.
Further, since $p\gamma<\beta$, by Lemma \ref{gt2015}(iii) with $\theta=0$, $a=N-\beta$, $b=p\gamma$ and $a+b<N$, we estimate
\begin{equation}\label{gamm01}
(I_\beta\ast u^p)u^{q}(x)\asymp \kappa^{p} u^q(x)=  \kappa^{p+q} |x|^{-\gamma q} \quad\mbox{ in }B_1\setminus\{0\}.
\end{equation}
Comparing \eqref{gamm2} and \eqref{gamm01} we see that for $\kappa\in (0,1)$ small enough, thanks to \eqref{gamm0} and $q>m-1$,
 one has that $u(x)=\kappa|x|^{-\gamma}$ is a singular positive solution of \eqref{main}.
}

\noindent (ii) Let $u$ be a positive singular solution of \eqref{main}. Using Propositions \ref{estg} and \ref{estko},
there exists $C>0$ such that for small $R>0$ we find
\begin{equation}\label{csigma}
C R^{-\sigma}\geq \sup_{|x|=R} u\geq cR^{-\frac{N-m-\alpha}{m-1}},
\end{equation}
where $\sigma>0$ is defined in \eqref{sig}. The above estimate implies $\sigma\geq \frac{N-m-\alpha}{m-1}$ which is equivalent to
$p+q\le (N+\beta)(m-1)/(N-m-\alpha)$.

We claim that both inequalities are strict. Assume by contradiction that
$\sigma= \frac{N-m-\alpha}{m-1}$ and let $x\in B_{1/2}\setminus\{0\}$. Combining the estimate \eqref{csigma} with the weak Harnack inequality \eqref{whi} with $a=1$, $b=1/2$, $c=1/8$ and $\ell=p>m-1$,  we find
$$
\begin{aligned}
(I_\beta\ast u^p)(x)&\geq \int_{B_{5|x|/4}\setminus B_{|x|/4}}\frac{u^p(y)}{|x-y|^{N-\beta}}dy\\
&\geq \Big(\frac{9|x|}{4}  \Big)^{\beta-N} \int_{B_{5|x|/4}\setminus B_{|x|/4}}u^p(y)dy\\
&\geq C|x|^{\beta-N}\Big(|x|^{N/p}\sup_{\partial B_{|x|}}u\Big)^p && \qquad\mbox{(by Harnack's inequality \eqref{whi})}\\
&\geq C|x|^{\beta-\sigma p} &&\qquad\mbox{ (by estimate \eqref{csigma})}.
\end{aligned}
$$
Hence, $u$ satisfies
$$
{\rm div}\big(|x|^{-\alpha}|\nabla u|^{m-2}\nabla u\big)\geq C|x|^{\beta-\sigma p}u^q \qquad \mbox{ in } B_{1/2}\setminus\{0\}.
$$
For any $k\geq 3$ let $v_k\in C^1(B_{1/2}\setminus B_{1/k})$ be a radial function such that
$$
\left\{
\begin{aligned}
{\rm div}\big(|x|^{-\alpha}|\nabla v_k|^{m-2}\nabla v_k\big)&=C|x|^{\beta-\sigma p}v_k^q  &&\quad\mbox{ in }B_{1/2}\setminus\overline B_{1/k},\\
v_k&=\sup_{|x|=1/k} u &&\quad\mbox{ on }\partial B_{1/k},\\
v_k&=\sup_{|x|=1/2} u &&\quad\mbox{ on }\partial B_{1/2}.
\end{aligned}
\right.
$$
Observe that $u$ is a subsolution while $c\Phi_{m,\alpha}$ is a supersolution of the above problem for suitable $c>0$.
By the maximum principle we find that $k\longmapsto v_k$ is increasing and
\begin{equation}\label{equu1}
c\Phi_{m,\alpha}\geq v_k\geq u\quad\mbox{ in }B_{1/2}\setminus B_{1/k},
\end{equation}
for some constant $c>0$.
Thus, there exists $v(x)=\lim_{k\to \infty}v_k(x)$ for all $x\in \overline{B}_{1/2}\setminus\{0\}$ and $v\in C^1(B_{1/2}\setminus\{0\})$ satisfies
\begin{equation}\label{equu2}
{\rm div}\big(|x|^{-\alpha}|\nabla v|^{m-2}\nabla v\big)=C|x|^{\beta-\sigma p} v^q \quad\mbox{ in }B_1\setminus\{0\}.
\end{equation}
Also $v$ is radial (since $v_k$ is radial) and from \eqref{equu1} we find
\begin{equation}\label{equu3}
c\Phi_{m,\alpha}\geq v\geq u\quad\mbox{ in }B_{1/2}\setminus \{0\}.
\end{equation}
Using this inequality it is easy to see that $v$ satisfies the conditions of Proposition \ref{strongh}
with
$$
a(x)=|x|^{\beta-\sigma p}v(x)^{q-m+1}\le c |x|^{\beta-\sigma(p+q-m+1)}=c|x|^{-m-\alpha}.
$$
Thus, by \eqref{harns}, \eqref{csigma} and \eqref{equu3} we find
$$
v(x)\geq c\sup_{|y|=|x|}v(y)\geq C|x|^{-\sigma}\quad\mbox{ for all }x\in B_{1/4}\setminus\{0\}.
$$
From  \eqref{equu2} and the above estimate we find
$$
\Big(r^{N-1-\alpha}|v'(r)|^{m-2}v'(r)\Big)'=Cr^{N-1+\beta-\sigma p} v^q\geq C r^{N-1+\beta-\sigma(p+q)} \quad\mbox{ for all } 0<r<1/4.
$$
Since
$$
\sigma= \frac{N-m-\alpha}{m-1}=\frac{m+\alpha+\beta}{p+q-m+1},
$$
the above estimate reads
$$
\Big(r^{N-1-\alpha}|v'(r)|^{m-2}v'(r)\Big)'\geq Cr^{-1}\quad\mbox{ for all } 0<r<1/4.
$$
We now fix $\bar{r}\in (0, 1/4)$ and integrate in the above inequality over $[r,\bar{r}]$. We obtain
$$
-r^{N-1-\alpha}|v'(r)|^{m-2}v'(r)\geq -\bar{r}^{N-1-\alpha}|v'(\bar{r})|^{m-2}v'(\bar{r})+C\ln\frac{\bar{r}}{r} \quad\mbox{ for all } 0<r<\bar{r}<1/4.
$$
From here we have
$$
\lim_{r\to 0^+} r^{N-1-\alpha}|v'(r)|^{m-2}v'(r)=-\infty,
$$
so that
$$
\lim_{r\to 0^+} \frac{v'(r)}{r^{-\frac{N-1-\alpha}{m-1}}}=-\infty.
$$
By l'Hopital's rule it follows that
$$
\lim_{r\to 0^+} \frac{v(r)}{\Phi_{m,\alpha}(r)}=\infty,
$$
which contradicts \eqref{equu3} and proves our claim. Hence
$\sigma> \frac{N-m-\alpha}{m-1}$ which yields \eqref{cond}$_2$. Also, \eqref{cond}$_3$ follows from \eqref{cond}$_2$ and the fact that $p,q>m-1$.

To derive the first inequality in \eqref{cond} we combine the weak Harnack inequality and Proposition \ref{estg} with the regularity condition $u\in L^p(B_1)$. We find

$$
\begin{aligned}
\infty &>\int_{B_{1/2}} u^p\geq \sum_{k=1}^\infty \int_{2^{-1-3k}<|y|<2^{2-3k}}u^p(y)dy&&\\
&\geq \frac{C}{2^N} \sum_{k=1}^\infty 2^{-3kN} \Big(\sup_{\frac{5}{2}\cdot  2^{-3k}<|y|<\frac{7}{2}\cdot  2^{-3k}} u(y)\Big)^p&& \mbox{ (by \eqref{whi} with $R=2^{-1-3k}$, $a=7$, $b=5$, $c=1/2$)}\\
&\geq \frac{C}{2^N} \sum_{k=1}^\infty 2^{-3kN} \Big(\sup_{ |y|=3\cdot2^{-3k}} u(y)\Big)^p\\
&\geq \frac{C}{2^N \cdot 3^{\frac{N-m-\alpha}{m-1}p}}   \sum_{k=1}^\infty \frac{1}{(8^{N-\frac{N-m-\alpha}{m-1}p})^{k}}.
&& \mbox{ (by Proposition \ref{estg})}
\end{aligned}
$$
This implies $N-\frac{N-m-\alpha}{m-1}p>0$ which establishes the first inequality in \eqref{cond} for $p$.

If $q\geq \frac{N(m-1)}{N-m-\alpha}$, by Proposition \ref{vv}  we deduce $u\in L^\infty(B_1)$, which is not possible since $u$ is singular.
Hence, $q<\frac{N(m-1)}{N-m-\alpha}$.

Conversely, assume that \eqref{cond} holds.
We construct a singular radially symmetric solution $u$ of \eqref{main} in the form $u(x)=\kappa|x|^{-\gamma}$,
with $\kappa,\gamma>0$ to be determined.

\noindent{\bf Case 1: } $\sigma p>\beta$.

Let
\begin{equation}\label{gamm}
\max\Big\{\frac{N-m-\alpha}{m-1}, \frac\beta p\Big\}<\gamma<\min\left\{\sigma, \frac{N}{p}  \right\}.
\end{equation}
Note that this choice of $\gamma$ is possible thanks to \eqref{cond}$_1$ and to our assumption $\sigma>\frac{N-m-\alpha}{m-1}$.
Also, $u(x)=\kappa|x|^{-\gamma}$ satisfies \eqref{gamm2}, where now the positivity of $A$ follows from the lower bound of $\gamma$.

By Lemma \ref{gt2015}(i) with $\theta=0$, $a=N-\beta$, $b=p\gamma$  so that $a+b>N$  being $p\gamma>\beta$, we find
\begin{equation}\label{gamm1}
(I_\beta\ast u^p)u^{q}(x)\leq C\kappa^{p}|x|^{\beta-p\gamma}u^q(x)\leq C \kappa^{p+q} |x|^{\beta-(p+q)\gamma} \quad\mbox{ in }B_1\setminus\{0\}.
\end{equation}

Using \eqref{gamm2},  \eqref{gamm1} and the fact that $p+q>m-1$ together with $\gamma<\sigma$, we may take $\kappa\in (0,1)$ small enough such that
$$
\begin{aligned}
{\rm div}\big(|x|^{-\alpha}|\nabla u|^{m-2}\nabla u\big)&=C_1\kappa^{m-1}|x|^{-(m-1)\gamma-m-\alpha}\\
&\geq C_2 \kappa^{p+q} |x|^{\beta-(p+q)\gamma}\\
& \geq (I_\beta\ast u^p)u^{q}(x)
\quad\mbox{ in }B_1\setminus\{0\}.
\end{aligned}
$$
This shows that $u(x)=\kappa|x|^{-\gamma}$ is a positive singular solution of \eqref{main} in $B_1\setminus\{0\}$.
\medskip

\noindent{\bf Case 2: } $\sigma p\leq \beta$.

Let us observe first that this condition is equivalent to
$$
\frac{m+\alpha}{q-m+1}\leq \sigma\leq \frac{\beta}{p}.
$$
Indeed, by replacing in $\sigma p\leq \beta$ the value of $\sigma$ given in \eqref{sig},   we get
$$(m+\alpha)p\le \beta (q-m+1).$$
Adding $(m+\alpha)(q-m+1)$ on both sides of the above inequality we find

\begin{equation}\label{ineq_1_1b}(m+\alpha)(p+q-m+1)\leq (m+\alpha+\beta)(q-m+1),\end{equation}
namely
\begin{equation}\label{ineq1b}\frac{m+\alpha}{q-m+1}\leq \frac{m+\alpha+\beta}{p+q-m+1}=\sigma,
\end{equation}
thus the required lower bound for $\sigma$ follows, as the upper bound trivially holds since we are in Case 1b.

Let
$$
\frac{N-m-\alpha}{m-1}<\gamma<\frac{m+\alpha}{q-m+1}\leq \sigma\leq \frac{\beta}{p}.
$$
(Note that from \eqref{cond}$_1$ we have $\frac{N-m-\alpha}{m-1}<\frac{m+\alpha}{q-m+1}$).
Letting $u(x)=\kappa |x|^{-\gamma}$, we have that $u$ satisfies \eqref{gamm2}, where here $A>0$ by the lower bound of $\gamma$.
Also, by Lemma \ref{gt2015}(iii) with $\theta=0$, $a=N-\beta$, $b=p\gamma$ so that $a+b<N$ being $\beta >p\gamma$, we have
\begin{equation}\label{gamm3}
(I_\beta\ast u^p)u^{q}(x)\leq C\kappa^{p} u^q(x)\leq C \kappa^{p+q} |x|^{-q\gamma} \quad\mbox{ in }B_1\setminus\{0\}.
\end{equation}
Combining \eqref{gamm2} and \eqref{gamm3} in the same way as we did in Case 1 we derive that $u(x)=\kappa |x|^{-\gamma}$ is a singular solution of \eqref{main}.

\qed

\section{Proof of Theorem \ref{thm2}}

(i) Any singular solution $u$ of \eqref{main1} fulfills in particular \eqref{main}. Thus, by Theorem \ref{thm1} conditions \eqref{cond} must hold if $N>m+\alpha$.

{
Conversely, assume now that either $N\leq m+\alpha$ or $N>m+\alpha$ and \eqref{cond} holds. Let  $\sigma$ be defined by \eqref{sig} and $\tau=\frac{1}{p+q-m+1}>0$.

We claim that
$$
u(x)=
\left\{
\begin{aligned}
&|x|^{-\sigma} &&\quad\mbox{ if }\sigma p>\beta,\\
&|x|^{-\sigma} \Big(\log\frac{5}{|x|}\Big)^{-\tau} &&\quad\mbox{ if }\sigma p=\beta,\\
&|x|^{-\frac{m+\alpha}{q-m+1}}&&\quad\mbox{ if }\sigma p<\beta,
\end{aligned}
\right.
$$
is a solution of \eqref{main1}. A straightforward calculation using Remark \ref{remk}
yields
$$
{\rm div}\big(|x|^{-\alpha}|\nabla u|^{m-2}\nabla u\big)\asymp
\left\{
\begin{aligned}
&|x|^{-\sigma(m-1)-m-\alpha} &&\quad\mbox{ if }\sigma p>\beta,\\
&|x|^{-\sigma(m-1)-m-\alpha} \Big(\log\frac{5}{|x|}\Big)^{-(m-1)\tau} &&\quad\mbox{ if }\sigma p=\beta,\\
&|x|^{-\frac{q(m+\alpha)}{q-m+1}}&&\quad\mbox{ if }\sigma p<\beta.
\end{aligned}
\right.
$$
To see this  we first note that \eqref{cond}$_2$ implies
\begin{equation}
\sigma>  \frac{N-m-\alpha}{m-1}.
\end{equation}
Thus, the coefficient $A$ defined in \eqref{abc} (in which $\gamma=\sigma$) satisfies $A>0$.

Also, by Lemma \ref{gt2015}(i)-(iii) (we use $\theta=\tau p\in (0,1)$ if $\sigma p=\beta$) we have
$$
(I_\beta\ast u^p)u^q(x)\asymp
\left\{
\begin{aligned}
&|x|^{\beta- \sigma(p+q)} &&\quad\mbox{ if }\sigma p>\beta,\\
&|x|^{-q\sigma} \Big(\log\frac{5}{|x|}\Big)^{1-\tau(p+q)} &&\quad\mbox{ if }\sigma p=\beta,\\
&|x|^{-\frac{q(m+\alpha)}{q-m+1}}&&\quad\mbox{ if }\sigma p<\beta,
\end{aligned}
\right.
$$
where in the latter case $\sigma p<\beta$, from Case 2   in the proof of Theorem \ref{thm1}(ii), we have that \eqref{ineq_1_1b}
holds with the strict sign so that we fall in Case(iii) of  Lemma \ref{gt2015}.}

From the above estimates we have
$$
{\rm div}\big(|x|^{-\alpha}|\nabla u|^{m-2}\nabla u\big)\asymp (I_\alpha\ast u^p)u^q
$$
and thus, for suitable constants $a\geq b>0$ we have that $u$ satisfies \eqref{main1}.

(ii) Let $u$ be a singular solution of \eqref{main1}. We divide our argument into two steps.

\noindent{\bf Step 1: $u$ satisfies the strong Harnack inequality \eqref{harns}.}

Note first that $u$ satisfies the inequality
$$
{\rm div}\big(|x|^{-\alpha}|\nabla u|^{m-2}\nabla u\big)\geq c u^{q}\quad\mbox{ in }B_1\setminus\{0\},
$$
where $c=2^{\alpha-N}\int_{B_1}u^p dx>0$.
Applying Proposition \ref{lem1} with $\theta=0$ we find
\begin{equation}\label{eqw8}
u(x)\leq C|x|^{-\frac{m+\alpha}{q-m+1}}\quad\mbox{ in }B_1\setminus\{0\}.
\end{equation}
Using the above estimate (if $\sigma p<\beta$) and \eqref{ttau} (if $\sigma p>\beta$), from   Lemma \ref{gt2015}(i),(iii) we obtain
\begin{equation}\label{eqw9}
(I_\beta\ast u^p)(x)\leq C\varphi(x) \quad\mbox{ in }B_1\setminus\{0\},
\end{equation}
where
\begin{equation}\label{eqw91}
\varphi(x)=|x|^{-(\sigma p-\beta)^+}.
\end{equation}
(we take $(\sigma p-\beta)^+=0$ if $\sigma p-\beta\leq 0$).
Now, \eqref{eqw9} together with \eqref{eqw8} (if $\sigma p<\beta$) and \eqref{ttau} (if $\sigma p>\beta$) imply
$$
(I_\beta\ast u^p) u^{q-m+1}\leq C|x|^{-m-\alpha}\quad\mbox{ in }B_1\setminus\{0\}.
$$
We are exactly in the frame of Proposition \ref{strongh} which yields \eqref{harns}.

\medskip

\noindent{\bf Step 2: Proof of \eqref{eqe1}-\eqref{eqe2}.}

Our analysis is split into two cases.

\noindent{\bf Case 1: } Suppose
\begin{equation}\label{suppp}
\limsup_{x\to 0}\frac{u(x)}{\Phi_{m,\alpha}(x)}<\infty.
\end{equation}
Let $c>0$ be such that $u(x)\leq c\Phi_{m,\alpha}(x)$ in $\overline B_1\setminus\{0\}$. By Lemma \ref{gt2015} we have
\begin{equation}\label{ia}
I_\beta\ast u^p\leq I_\beta\ast (c\Phi_{m,\alpha})^p\leq C|x|^{-\theta}\quad\mbox{ in }B_1\setminus\{0\},
\end{equation}
where
\begin{equation}\label{bett}
\theta=\left\{
\begin{aligned}
&p\frac{N-m-\alpha}{m-1}-\beta &&\quad\mbox{ if }\; p\frac{N-m-\alpha}{m-1}>\beta,\\
&\tau&&\quad\mbox{ if }\; p\frac{N-m-\alpha}{m-1}=\beta,\\
&0 &&\quad\mbox{ if }\; p\frac{N-m-\alpha}{m-1}<\beta,
\end{aligned}
\right.
\end{equation}
and $\tau>0$ is chosen small enough such that\footnote{We choose  $\tau$ with the above conditions just to avoid the log terms that appear in estimating the convolution integrals}
$$
q<\frac{N-(\sigma p-\beta)^+-\tau }{N-m-\alpha}(m-1).
$$
Also, by the definition \eqref{bett} of $\theta$ and \eqref{cond} we have $0\leq \theta<m+\alpha$,
this latter condition is required in the statement of Lemma \ref{lem2}.
Indeed, this is easy to check if $p\frac{N-m-\alpha}{m-1}\leq \beta$. If $p\frac{N-m-\alpha}{m-1}>\beta$ then we observe
that from \eqref{cond}$_2$ and $q>m-1$ we find
$$
p<\frac{m+\alpha+\beta}{N-m-\alpha}(m-1),
$$
and then
$$
\theta=p\frac{N-m-\alpha}{m-1}-\beta<m+\alpha.
$$

Since $u$ is a singular solution of \eqref{main1}, there exists
a decreasing sequence $\{r_k\}\subset (0,1)$, $r_k\to 0$ (as $k\to \infty$) such that
$$
\sup_{|x|=r_k}u(x) \to \infty\quad\mbox{ as }k\to \infty.
$$
Using the strong Harnack inequality \eqref{harns} we also have
\begin{equation}\label{infff}
\inf_{|x|=r_k} u(x) \to \infty\quad\mbox{ as }k\to \infty.
\end{equation}
For any $k\geq 1$ let $w_k\in C^1(B_1\setminus B_{r_k})$ be a radial function such that
$$
\left\{
\begin{aligned}
{\rm div}\big(|x|^{-\alpha}|\nabla w_k|^{m-2}\nabla w_k\big)&=C|x|^{-\theta} w_k^q &&\quad\mbox{ in }B_1\setminus\overline B_{r_k},\\
w_k&=\inf_{|x|=r_k} u &&\quad\mbox{ on }\partial B_{r_k},\\
w_k&=\inf_{|x|=1} u &&\quad\mbox{ on }\partial B_{1}.
\end{aligned}
\right.
$$
Since $u$ satisfies \eqref{ia}, by the maximum principle we find that $k\longmapsto w_k$ is increasing and $u\geq w_k$ in $B_1\setminus B_{r_k}$. Thus, there exists $w(x)=\lim_{k\to \infty}w_k(x)$ for all $x\in \overline{B}_1\setminus\{0\}$ and $w\in C^1(B_1\setminus\{0\})$ satisfies
$$
{\rm div}\big(|x|^{-\alpha}|\nabla w|^{m-2}\nabla w\big)=C|x|^{-\theta} w^q \quad\mbox{ in }B_1\setminus\{0\}.
$$
Also $w$ is singular since by \eqref{infff} we have
$$
\sup_{|x|=r_k} w\geq \sup_{|x|=r_k}w_k=\inf_{|x|=r_k} u \to \infty\quad\mbox{ as }k\to \infty.
$$
Thus, by \eqref{suppp} and Lemma \ref{lem2}  (which can be applied since in the first and in the second case of \eqref{bett},
condition \eqref{range_q_asy_beh} implies \eqref{range_q_theta} being $\sigma>(N-m-\alpha)/(m-1)$  by virtue of \eqref{cond}$_2$,
the third case of \eqref{bett} condition \eqref{range_q_asy_beh} is exactly \eqref{range_q_theta})
we deduce $u\asymp \Phi_{m,\alpha}$.

\noindent{\bf Case 2: } Suppose
$$
\limsup_{x\to 0}\frac{u(x)}{\Phi_{m,\alpha}(x)}=\infty.
$$
Hence, one may find a decreasing sequence $\{r_k\}\subset (0,1)$, $r_k\to 0$ (as $k\to \infty$) such that
$$
\sup_{|x|=r_k}\frac{u(x)}{\Phi_{m,\alpha}(x)}\to \infty\quad\mbox{ as }k\to \infty.
$$
Using the strong Harnack inequality \eqref{harns} for $u$ and the fact that $\Phi_{m,\alpha}(x)=\Phi_{m,\alpha}(|x|)$, one has
\begin{equation}\label{inff}
\inf_{|x|=r_k}\frac{u(x)}{\Phi_{m,\alpha}(x)}\to \infty\quad\mbox{ as }k\to \infty.
\end{equation}
Recall that $u$ satisfies \eqref{eqw9}-\eqref{eqw91}. For any $k\geq 1$ let $w_k\in C^1(B_1\setminus B_{r_k})$ be a radial function such that
$$
\left\{
\begin{aligned}
{\rm div}\big(|x|^{-\alpha}|\nabla w_k|^{m-2}\nabla w_k\big)&=C\varphi(x)w_k^q &&\quad\mbox{ in }B_1\setminus\overline B_{r_k},\\
w_k&=\inf_{|x|=r_k} u &&\quad\mbox{ on }\partial B_{r_k},\\
w_k&=\inf_{|x|=1} u &&\quad\mbox{ on }\partial B_{1},
\end{aligned}
\right.
$$
where $\varphi$ is defined in \eqref{eqw9}.
By the maximum principle we find that $k\longmapsto w_k$ is increasing and $u\geq w_k$ in $B_1\setminus B_{r_k}$.
 Thus, there exists $w(x)=\lim_{k\to \infty}w_k(x)$ for all $x\in \overline{B}_1\setminus\{0\}$ and
$$
{\rm div}\big(|x|^{-\alpha}|\nabla w|^{m-2}\nabla w\big)=C\varphi(x)w^q \quad\mbox{ in }B_1\setminus\{0\}.
$$
In particular, $w$ satisfies \eqref{eqte} with $\theta=(\sigma p-\beta)^+<m+\alpha$ since $q>m-1$, and
\begin{equation}\label{eqestt}
u\geq w\geq w_k\quad\mbox{ in }B_1\setminus B_{r_k}.
\end{equation}
Using the above estimates and \eqref{inff} it follows that
$$
\limsup_{x\to 0}\frac{ u(x)}{\Phi_{m,\alpha}(x)}\geq \lim_{k\to 0} \sup_{|x|=r_k} \frac{w(x)}{\Phi_{m,\alpha}(x)}\geq
\lim_{k\to 0} \sup_{|x|=r_k} \frac{w_k(x)}{\Phi_{m,\alpha}(x)}=\infty.
$$
By Lemma \ref{lem2} it follows that
$$
w\asymp |x|^{-\frac{m+\alpha-(\sigma p-\beta)^+}{q-m+1}}.
$$
This fact combined with \eqref{eqw8} (if $\sigma p<\beta$) and \eqref{ttau} (if $\sigma p>\beta$) implies the estimates in Theorem \ref{thm2}(ii).

\bigskip

\noindent
{\bf Aknowledgements.} RF was partly supported by the Italian MIUR project
{\em Variational methods, with applications to problems in mathematical physics and
geometry} (2015KB9WPT\_009) and by {\em Fondo Ricerca di
Base di Ateneo Esercizio} 2017-19 of the University of Perugia, named {\em Problemi
con non linearit\`a dipendenti dal gradiente}.
RF is a  member of the {\em Gruppo Nazionale per
l'Analisi Ma\-te\-ma\-ti\-ca, la Probabilit\`a e le loro Applicazioni}
(GNAMPA) of the {\em Istituto Nazionale di Alta Matematica} (INdAM).

\section*{Appendix: Proof of Lemma \ref{gt2015} }

In this section we present the proof of Lemma \ref{gt2015} which is rather technical.
For reader's convenience we include all details. We first establish the lower bound in the estimates \eqref{leqe1}-\eqref{leqe3}, that is,
$$
\int_{|y|<1} \frac{\Big( \log\frac{5}{|y|}\Big)^{-\theta} dy}{|x-y|^a|y|^b}  \geq
c\left\{
\begin{aligned}
&\frac{\Big(\displaystyle \log\frac{5}{|x|}\Big)^{-\theta}}{|x|^{a+b-N}}&&\quad\mbox{ if }a+b>N,\\
&\Big(\displaystyle \log\frac{5}{|x|}\Big)^{1-\theta} &&\quad\mbox{ if }a+b=N\mbox{ and }{\theta\neq 1},\\
&1 &&\quad\mbox{ if }a+b<N.
\end{aligned}
\right.
$$
It is enough to establish the above inequality for all $x\in B_{1/2}\setminus\{0\}$. Then, since all involved functions
are continuous on $\overline B_1\setminus B_{1/2}$ we may take a smaller constant $c>0$ such that the above estimate still
holds for all $x\in B_1\setminus\{0\}$.

Indeed, if we denote by $\phi(x)$ the function on RHS of the above estimate, then, for $1/2\leq |x|<1$ we have (since $|x-y|\leq |x|+|y|<2$)
$$
\int_{|y|<1} \frac{\Big( \log\frac{5}{|y|}\Big)^{-\theta} dy}{|x-y|^a|y|^b}\geq
\int_{|y|<1} \frac{\Big( \log\frac{5}{|y|}\Big)^{-\theta} dy}{2^a|y|^b}
= C_1(=const.)\geq \frac{C_1}{C_2}\phi(x),
$$
where
$$
C_2=\max_{1/2\leq |x|\leq 1}\phi(x).
$$
This shows that the inequality holds true on $B_1\setminus B_{1/2}$ so we need only to prove it on $B_{1/2}\setminus\{0\}$.

\bigskip

Observe that
$$
\begin{aligned}
\int_{|y|<1} \frac{\Big( \log\frac{5}{|y|}\Big)^{-\theta} dy}{|x-y|^a|y|^b} & \geq \int_{|x|<|y|<1} \frac{\Big( \log\frac{5}{|y|}\Big)^{-\theta} dy}{|x-y|^a|y|^b}\\
& \geq \int_{|x|<|y|<1} \frac{\Big( \log\frac{5}{|y|}\Big)^{-\theta} dy}{(2|y|)^a|y|^b}\\
&=\omega_N 2^{-a}\int_{|x|}^1  t^{N-a-b}\Big(\log\frac{5}{t}\Big)^{-\theta}\frac{dt}{t},
\end{aligned}
$$
where $\omega_N$ is the surface area of the unit ball in $\R^N$.
From here we estimate as follows:
\bigskip

(i1) If $a+b>N$ then
$$
\int_{|x|}^1  t^{N-a-b}\Big(\log\frac{5}{t}\Big)^{-\theta}\frac{dt}{t}
\geq
\Big(\log\frac{5}{|x|}\Big)^{-\theta}
\int_{|x|}^1  t^{N-a-b}\frac{dt}{t}\geq  c\,\frac{\Big(\log\frac{5}{|x|}\Big)^{-\theta}}{|x|^{a+b-N}},
$$
if $0<|x|<1/2$, with $c=\frac{1-2^{N-a-b}}{a+b-N}$.

\bigskip

(i2) If $a+b=N$ then, for any $0<|x|<1/2$ we have
$$
\begin{aligned}
\int_{|x|}^1  t^{N-a-b}\Big(\log\frac{5}{t}\Big)^{-\theta}\frac{dt}{t}&=
\int_{|x|}^1  \Big(\log\frac{5}{t}\Big)^{-\theta}\frac{dt}{t}\\
&=\left\{
\begin{aligned}&\frac{1}{1-\theta}\Big[\Big(\log \frac{5}{|x|}\Big)^{1-\theta}-
\Big(\log 5\Big)^{1-\theta}\Big]&&\quad\mbox{ if } 0\leq \theta\neq 1\\
&\log\biggl(\frac 1{\log 5}\cdot\log \frac{5}{|x|}\biggr)&&\quad\mbox{ if }  \theta=1 \end{aligned}\right.
\\&\geq c
\left\{
\begin{aligned}
&\Big(\log \frac{5}{|x|}\Big)^{1-\theta} &&\quad\mbox{ if } 0\leq \theta<1,\\
&1 &&\quad\mbox{ if }\theta \ge1.
\end{aligned}
\right.
\end{aligned}
$$
Indeed, if $\theta=1$  then
$$\int_{|x|}^1  \Big(\log\frac{5}{t}\Big)^{-\theta}\frac{dt}{t}=\log\biggl(\frac 1{\log 5}\cdot\log \frac{5}{|x|}\biggr)\ge
 \log\biggl(\frac{\log 10}{\log5}\biggr), \qquad 0<|x|<\frac12,$$
 while for $\theta>1$ we have
$$\frac{1}{1-\theta}\Big[\biggl(\log \frac{5}{|x|}\biggr)^{1-\theta}-
(\log 5)^{1-\theta}\Big]\ge \frac{(\log 5)^{1-\theta}-(\log 10)^{1-\theta} }{\theta-1}>0;$$
finally if $0\le\theta<1$ we have
$$\frac{\Big(\log {5}/{|x|}\Big)^{1-\theta}}{1-\theta}\Big[1-
\Big(\frac{\log 5}{\log {5}/{|x|}}\Big)^{1-\theta}\Big]\ge \frac{(\log 10)^{1-\theta}}{1-\theta}\Big[1-
\Big(\frac{\log 5}{\log 10}\Big)^{1-\theta}\Big]>0.$$
\bigskip

(i3) If $a+b<N$, and  $0<|x|<1/2$ we have
$$
\int_{|x|}^1  t^{N-a-b}\Big(\log\frac{5}{t}\Big)^{-\theta}\frac{dt}{t}\geq
\int_{1/2}^1  t^{N-a-b}\Big(\log\frac{5}{t}\Big)^{-\theta}\frac{dt}{t}\geq
 \frac{1-2^{a+b-N}}{N-a-b}(\log 10)^{-\theta}>0.$$

In order to establish the upper bounds in the estimates \eqref{leqe1}-\eqref{leqe3} we proceed as in \cite[Lemma 3.6]{GT2015} (see also \cite[Lemma 10.4]{GT2016book}).
 Let $r=|x|\in (0,1)$ and use the the change of variables $x=r\zeta$, $y=r\eta$. In particular, we have $|\zeta|=1$. Thus

$$
\begin{aligned}
\int_{|y|<1} \frac{\Big( \log\frac{5}{|y|}\Big)^{-\theta} dy}{|x-y|^a|y|^b} & \leq \int_{|y|<2} \frac{\Big(\log\frac{5}{|y|}\Big)^{-\theta} dy}{|x-y|^a|y|^b} \\
&=r^{N-a-b}  \int_{|\eta|<2/r} \frac{\Big(\ \log\frac{5}{r|\eta|}\Big)^{-\theta} d\eta}{|\zeta-\eta|^a|\eta|^b}\\
&\leq r^{N-a-b}\left[\Big(\displaystyle \log\frac{5}{2r}\Big)^{-\theta} \int_{0<|\eta|<2} \frac{d\eta}{|\zeta-\eta|^a|\eta|^b}
+\int_{2<|\eta|<2/r} \frac{\Big( \log\frac{5}{r|\eta|}\Big)^{-\theta} d\eta}{|\zeta-\eta|^a|\eta|^b}\right]\\
&\leq r^{N-a-b}\left[A\Big(\displaystyle \log\frac{5}{2r}\Big)^{-\theta} +\int_{2<|\eta|<2/r}
\frac{\Big( \log\frac{5}{r|\eta|}\Big)^{-\theta}
 d\eta}{(|\eta|/2)^a|\eta|^b}\right]\\
\end{aligned}
$$
where
$$
A=\max_{|\zeta|=1} \int_{0<|\eta|<2} \frac{d\eta}{|\zeta-\eta|^a|\eta|^b}  \in (0, \infty),
$$
and where we have used the trivial property $|\zeta-\eta|^a\ge ||\zeta|-|\eta||^a=(|\eta|-1)^a\ge (|\eta|/2)^a$, when $|\eta|>2$.

By virtue of
$$\frac{\log 5/(2r)}{\log 5/r}\ge 1-\frac{\log2}{\log5} ,\quad 0<r<1,$$
we immediately derive
$$\Big(\displaystyle \log\frac{5}{2r}\Big)^{-\theta} \le c \Big(\displaystyle \log\frac{5}{r}\Big)^{-\theta}, \quad 0<r<1,$$
so that
$$
\begin{aligned}
\displaystyle \int\limits_{|y|<1} \frac{\Big( \log\frac{5}{|y|}\Big)^{-\theta} dy}{|x-y|^a|y|^b} & \leq Cr^{N-a-b}\left[\Big(\displaystyle \log\frac{5}{r}\Big)^{-\theta} +\int_{2}^{2/r} t^{N-a-b} \Big(\log\frac{5}{r t}\Big)^{-\theta} \frac{dt}{t}\right]\\&:=Cr^{N-a-b} I(r,\theta).
\end{aligned}
$$
Next, a straightforward calculation leads to the desired estimates in the upper bounds of (i)-(iii).
Indeed, we proceed as follows.

(ii1) If $a+b>N$, the upper bound in \eqref{leqe1} is in force, because choosing $\varepsilon\in (0,a+b-N)$ we obtain

$$\int_{2}^{2/r} t^{N-a-b} \Big(\log\frac{5}{r t}\Big)^{-\theta} \frac{dt}{t^{1-\varepsilon+\varepsilon}}\le c \frac{2^{N-a-b+\varepsilon}}{a+b-\varepsilon-N}\biggl(1-r^{a+b-\varepsilon-N}\biggr) \Big(\log\frac{5}{r}\Big)^{-\theta}\le c \Big(\log\frac{5}{r}\Big)^{-\theta},$$
where we have used that there exists $c>0$ such that
$$t^{-\varepsilon}\biggl(\lg\frac 5{rt}\biggr)^{-\theta}\le c\biggl(\log\frac 5r\biggr)^{-\theta}, \quad\mbox{for all } t\in \biggl(2,\frac2r\biggr), \quad r\in(0,1).$$

(ii2) If $a+b=N$ then
$$\int_{2}^{2/r} t^{N-a-b} \Big(\log\frac{5}{r t}\Big)^{-\theta} \frac{dt}{t}=
\begin{cases}\dfrac 1{1-\theta}\Big[\bigl(\log\frac{5}{2r}\bigr)^{1-\theta}-\bigl(\log\frac{5}{2}\bigr)^{1-\theta}\Big]&\quad\mbox{if }0\le \theta\neq1\\
\log\log\frac5r-\log\log\frac52&\quad\mbox{if }\theta=1
\end{cases}
$$
Consequently, also the upper bound in \eqref{leqe2} holds, since, if $\theta>1$, using that
$$\log \frac5r>\log 5, \qquad \log \frac 5{2r} > \log \frac52, \qquad 0<r<1,$$ we have
$$I(r,\theta)\le (\log 5)^{-\theta}+\frac1{\theta-1}\biggl[\biggl(\log\frac{5}{2}\biggr)^{1-\theta}-(\log 5)^{1-\theta}\biggr],
$$
while if $0\le\theta<1$, we arrive to
$$I(r,\theta)\le\biggl[\biggl(\log\frac{5}{r}\biggr)^{-1}+\frac 1{1-\theta}\biggr]\biggl(\log\frac5r\bigr)^{1-\theta}-\frac1{1-\theta}
\biggl(\log\frac52\biggr)^{1-\theta}\le \biggl[\frac 1{\log5} +\frac 1{1-\theta}\biggr]\biggl(\log\frac5r\bigr)^{1-\theta}.
$$

(ii3) When $a+b<N$, the upper bound in \eqref{leqe3} follows immediately since
$$r^{N-a-b} I(r,\theta)\le r^{N-a-b}(\log 5)^{-\theta} + \biggl(\log\frac52\biggr)^{-\theta}\frac{2^{N-a-b}}{N-a-b} \bigl(1-r^{N-a-b}\bigr)\le C$$
for all $0<r<1$.

\qed

\end{document}